\documentclass[10pt]{amsart}
\usepackage{ifpdf}
\usepackage[colorlinks,linkcolor=blue,citecolor=blue,urlcolor=blue, pdfcenterwindow, pdfstartview={XYZ null null 1.2}, pdffitwindow,pdfdisplaydoctitle=true]{hyperref}
\usepackage{hyperref}
\usepackage[all]{hypcap}
\usepackage[english]{babel}
\usepackage{array}
\usepackage{threeparttable}
\usepackage{rotating}
\usepackage[section]{placeins}
\usepackage{amssymb} 
\usepackage{amsmath}
\usepackage{mathrsfs}
\usepackage{breqn}
\usepackage{bbm}
\usepackage{enumerate}
\usepackage{color}
\usepackage{relsize}
\usepackage{amsrefs} 

\newtheorem{lemma}{Lemma}[section]

\newtheorem{prop}[lemma]{Proposition}
\newtheorem{coro}[lemma]{Corollary}
\newtheorem{theo}[lemma]{Theorem}

\newtheorem{rema}[lemma]{Remark}

\newtheorem{defi}[lemma]{Definition}

\newcommand{\sss}{\mathbb{S}}
\newcommand{\ttt}{\mathbb{T}}
\newcommand{\ungra}{\textbf{1}}

\DeclareMathOperator{\SO}{SO}
\DeclareMathOperator{\SU}{SU}

\title[Exact convergence rate in the ergodic theorem]{The exact convergence rate in the ergodic theorem of Lubotzky-Phillips-Sarnak}

\author{A. Pinochet Lobos}
\address{Aix Marseille Universit\'e, CNRS, Centrale Marseille, I2M, Marseille, France}
\email{antoine.pinochet-lobos@univ-amu.fr}
\author{Ch. Pittet}
\address{Aix Marseille Universit\'e, CNRS, Centrale Marseille, I2M, Marseille, France,
 and Section de math\'ematiques, Universit\'e de Gen\`eve, Gen\`eve, Suisse}
\email{pittet@math.cnrs.fr}



\thanks{This research was conducted during a CRCT of the second author.}

\date{October 17th, 2018}

\begin{document}

\maketitle

\begin{abstract} We compute exact convergence rates in von Neumann type ergodic theorems when the acting group of measure-preserving transformations is free and the means are taken over spheres or over balls defined by a word metric. Relying on the upper bounds on the spectra of Koopman operators deduced by Lubotzky, Phillips, and Sarnak, from Deligne's work on the Weil conjecture, we compute the exact convergence rate for the free groups (of rank $(p+1)/2$ where $p\equiv 1\mod 4$ is prime) of isometries of the round sphere, defined by Lipschitz quaternions. We also show that any finite rank free group of automorphisms of the torus, realizes the lowest possible discrepancy, and prove a matching upper bound on the convergence rate. 
\end{abstract} 

\tableofcontents

\section{Introduction}

\subsection{Convergence rates on the sphere}
In  \cite{LPS-1} and \cite{LPS-2} the  theory of automorphic forms and the theory of unitary representations are applied to compute the discrepancy of orbit points of Lipschitz quaternions on the $2$-sphere. More precisely, let $q=x_0+x_1i+x_2j+x_3k$ be a quaternion and let
$N(q)=x_0^2+x_1^2+x_2^2+x_3^2$ be its norm. let $p$ be a prime such that $p\equiv1\mod 4$. Let $\Sigma_{p+1}\subset \SO(3,\mathbb R)$ denote the image of the set of Lipschitz quaternions
\[
\{q=x_0+x_1i+x_2j+x_3k: x_0,x_1,x_2,x_3\in\mathbb Z: N(q)=p: x_0>0: x_0\equiv1\mod 2\}
\]
under the adjoint representation. Let $\nu$ be the probability Lebesgue measure on the round sphere $\sss^2$.
\begin{theo}(Lubotzky-Phillips-Sarnak \cite[Theorem 1.3, Theorem 1.5]{LPS-1}.)
The subgroup $\Gamma$ of $\SO(3,\mathbb R)$ generated by the symmetric set $\Sigma_{p+1}$ is free of rank $\frac{p+1}{2}$ and 
\[
\sup_{\|f\|_2=1}\left\|x\mapsto\left(\frac{1}{|\Sigma_{p+1}|}\sum_{\gamma\in \Sigma_{p+1}}f(\gamma x)-\int_{\sss^2}f(y)d\nu(y)\right)\right\|_2=\frac{2 \sqrt p}{p+1}.
\]
Let $E_n$ be either the sphere  or the ball of $\Gamma$ around $e\in\Gamma$ of radius $n$ with respect to the word metric defined by  $\Sigma_{p+1}$. There is a constant $C>0$ such that for all $n\in\mathbb N$,
\[
\sup_{\|f\|_2=1}\left\|x\mapsto\left(\frac{1}{|E_n|}\sum_{\gamma\in E_n}f(\gamma x)-\int_{\sss^2}f(y)d\nu(y)\right)\right\|_2\leq Cnp^{-n/2}.
\]
\end{theo}
In the above statement, the expression $\|x\mapsto \varphi(x)\|_2$ denotes the $L^2$-norm $\|\varphi\|_2$ of a function $\varphi$. In what follow we will often shorten the notation, writing  $\|\varphi(x)\|_2$ instead of $\|x\mapsto \varphi(x)\|_2$.  

The next theorem is our main result. It generalizes \cite[Theorem 1.3]{LPS-1} and strengthens \cite[Theorem 1.5]{LPS-1}. 

\begin{theo}\label{thm: main}
Let $\Gamma$ be the free subgroup of rank $\frac{p+1}{2}$ of $\SO(3,\mathbb R)$ generated by the symmetric set $\Sigma_{p+1}$. Let $S_n$, respectively $B_n$, be the sphere, respectively the ball, of $\Gamma$ around $e\in\Gamma$ of radius $n$ with respect to the word metric defined by  $\Sigma_{p+1}$. Then
\[
\sup_{\|f\|_2=1}\left\|\frac{1}{|S_n|}\sum_{\gamma\in S_n}f(\gamma x)-\int_{\sss^2}f(y)d\nu(y)\right\|_2=\left(1+\frac{p-1}{p+1}n\right)p^{-n/2},
\]

\[
\sup_{\|f\|_2=1}\left\|\frac{1}{|B_n|}\sum_{\gamma\in B_n}f(\gamma x)-\int_{\sss^2}f(y)d\nu(y)\right\|_2=c(p,n)\left(1+\left(1+\frac{1}{\sqrt p}\right)n\right)p^{-n/2},
\]
where $c(p,n)=\frac{p-1}{p+1-\frac{2}{p^n}}$.
\end{theo}

The result is proved by establishing matching upper and lower bounds on the discrepancy (for the definition of the discrepancy, see Formula \ref{formula: discrepancy} below).

\subsection{Upper bounds: Weil and Deligne}
 The upper bounds are obtained with the help of three main ingredients. The first ingredient is the inclusion of the spectrum of a Koopman operator, associated to the free subgroup of $\SO(3,\mathbb R)$  generated by Lipschitz quaternions, and defined by an Hecke element, in the spectrum of the corresponding operator defined by the regular representation. See Formula \ref{formula: inclusion between spectra} below for the precise statement. To the best of our knowledge, the only known proof of this inclusion, is the one from \cite[S153-S158]{LPS-1} and \cite[Theorem 4.1]{LPS-2}, which uses the theory of automorphic forms and Deligne's solution to the Weil conjecture. The second ingredient is an application of the spectral theorem to Hecke elements. The third one is an identity between the norm of operators defined by the regular representation of a free group and values of the Harish-Chandra function of a free group: see Proposition \ref{prop: computing the norm of the regular representation} below. 

\subsection{Lower bounds: a general fact}
The lower bounds follow from a general result of Shalom announced in \cite[Theorem 4.14]{ShaTata} and also stated in 
\cite[Proposition 7]{DudGri}: if a finitely generated group $\Gamma$ acts by measure-preserving transformations on an atomless probability space $(X,\nu)$, then there is a subgroup $H$ of $\Gamma$, such that the quasi-regular representation of $\Gamma$  on $l^2(\Gamma/H)$, is weakly contained in the restriction of the Koopman representation of $\Gamma $ to the orthogonal complement $L^2_0(X,\nu)$ of the constant functions. 
In the following proposition, we spell-out the consequence of this result we need. 

\begin{prop}\label{prop: H-Ch lower bound on the discrepancy for free groups}
Assume that $\Gamma$ is a free group of rank $r\geq 1$ acting by measure-preserving transformations on an atomless probability space $(X,\nu)$. Let $\{a_1,\dots,a_r\}$ be a free generating set of $\Gamma$. Let $S=\{a_1^{\pm 1},\dots,a_r^{\pm 1}\}$. Let $q=2r-1$. 
Let $S_n$, respectively $B_n$, be the sphere, respectively the ball, of $\Gamma$ around $e\in\Gamma$ of radius $n$ with respect to the word metric defined by  $S$. Then
\[
\sup_{\|f\|_2=1}\left\|\frac{1}{|S_n|}\sum_{\gamma\in S_n}f(\gamma x)-\int_Xf(y)d\nu(y)\right\|_2\geq\left(1+\frac{q-1}{q+1}n\right)q^{-n/2},
\]
\[
\sup_{\|f\|_2=1}\left\|\frac{1}{|B_n|}\sum_{\gamma\in B_n}f(\gamma x)-\int_Xf(y)d\nu(y)\right\|_2\geq c(q,n)\left(1+\left(1+\frac{1}{\sqrt q}\right)n\right)q^{-n/2},
\]
where $c(q,n)=\left(1+2q^{-n}\sum_{k=0}^{n-1}q^k\right)^{-1}$.
\end{prop}

\begin{rema}\label{rema: amenable} Notice that both lower bounds evaluated at $q=1$ give the value $1$ and in this case both inequalities are equalities. When $q>1$, then $c(q,n)=\frac{q-1}{q+1 -\frac{2}{q^n}}$.
\end{rema}

\subsection{Convergence rates on the torus}
  
It follows from \cite[Theorem 1.4]{LPS-1} that  a generic finitely generated free subgroup of $\SO(3,\mathbb R)$ does \emph{not} realize the lower bounds
of Proposition \ref{prop: H-Ch lower bound on the discrepancy for free groups}. This is in contrast with the group of automorphisms of the torus where any finitely generated free subgroup realizes the fastest possible convergence rate, as stated in the following theorem which easily follows from \cite[Theorem 4.17]{ShaTata} or \cite[20]{dlH}, or ideas presented in \cite{Fin}, or \cite{GabGal}.

\begin{theo}\label{thm: exact convegence rate on the torus} Let $\ttt^2=\mathbb R^2/\mathbb Z^2$ be the $2$-torus with its normalized Haar measure $\nu$ and let $GL(2,\mathbb Z)\simeq \mbox{Aut}(\ttt^2)=G$ be its automorphism group.
	Assume that $\Gamma<G$ is a free subgroup of rank $r\geq 1$ freely generated by $\{a_1,\dots,a_r\}\subset G$. Let $S=\{a_1^{\pm 1},\dots,a_r^{\pm 1}\}$. Let $q=2r-1$. 
	Let $S_n$, respectively $B_n$, be the sphere, respectively the ball, of $\Gamma$ around $e\in\Gamma$ of radius $n$ with respect to the word metric defined by  $S$. Then
	\[
	\sup_{\|f\|_2=1}\left\|\frac{1}{|S_n|}\sum_{\gamma\in S_n}f(\gamma x)-\int_{\ttt^2}f(y)d\nu(y)\right\|_2=\left(1+\frac{q-1}{q+1}n\right)q^{-n/2},
	\]

	\[
	\sup_{\|f\|_2=1}\left\|\frac{1}{|B_n|}\sum_{\gamma\in B_n}f(\gamma x)-\int_{\ttt^2}f(y)d\nu(y)\right\|_2= c(q,n)\left(1+\left(1+\frac{1}{\sqrt q}\right)n\right)q^{-n/2},
	\]
	where $c(q,n)=\left(1+2q^{-n}\sum_{k=0}^{n-1}q^k\right)^{-1}$.
\end{theo}

\subsection{Related references}
 We close this introduction by mentioning several works related to  \cite{LPS-1} and \cite{LPS-2}. The five pages paper of Arnol'd and Krylov \cite{ArnKry} is one of the earlier reference on free groups of rotations. Lubotzky's book \cite{Lub} (specially Chapter 9) is a general reference to the subject  and its numerous ramifications. Colin de Verdi\`ere has given a s\'eminaire Bourbaki \cite{Colin} on \cite{LPS-1}, \cite{LPS-2}.  Shalom's  survey \cite{ShaTata} presents estimates of the discrepancy of random points. In a series of paper, Bourgain and Gamburd (see \cite{BouGan} and references therein)  construct finite symmetric sets in $\SU(d)$ whose Koopman's operators have norms smaller than $1$. Clozel \cite{Clo} has obtained sharp  bounds (up to multiplicative constants) on the discrepancies of some subsets of $\SO(2n)$. In \cite{COU}, Clozel, Oh, and Ullmo, express convergence rates for ergodic theorems on locally symmetric spaces, in terms of Harish-Chandra functions. The survey \cite{GorNevBul} of Gorodnik and Nevo includes many convergence rates estimates. In \cite{EMV}, Ellenberg, Michel, Venkatesh, discuss and improve Linnik's work on the distribution of the spatial distribution of point sets on the $2$-sphere, obtained from the representation of a large integer as a sum of three integer squares. In \cite{BRS1} and \cite{BRS2}, Bourgain, Rudnick, Sarnak, evaluate this distribution through different ``statistics'', and compare it with what's happening in higher dimension, and with the case of random points. Parzanchevski and Sarnak have shown  that the optimal generating rotations $\Sigma_{p+1}\subset \SO(3,\mathbb R)$ can be used to construct efficient gates needed as building blocks for quantum algorithms \cite{ParSar}. 

\subsection{Acknowledgements}
We are very grateful to Pierre de la Harpe who's suggestions and remarks on an earlier version enabled us to substantially improve the quality of the exposition. We also warmly thank Tatiana Smirnova-Nagnibeda, Pascal Hubert, Anders Karlsson, Arnaldo Nogueira, and Alain Valette for their interest in this work.  We thank Alex Lubotzky for pointing out to us reference \cite{ParSar}. We are grateful
to Pierre-Emmanuel Caprace, for suggesting to us a generalization of Theorem \ref{theo: lower bound on discrepancy} below, to the case of locally compact groups (see Remark 
\ref{rem: Pierre-Emmanuel} below).

\section{Discrepancy and the Koopman representation}

In this section we recall the relevant  definitions and notation from representation theory needed for the proofs of the results stated in the introduction. 
 
\subsection{Three involutive algebras}
Let $\Gamma$ be a group. Recall that the formal linear combinations  of elements of $\Gamma$ with complex coefficients
\[
	\sum_{\gamma\in\Gamma}a_{\gamma}\gamma,
\]
form the group algebra $\mathbb C[\Gamma]$ of $\Gamma$ over the complex numbers. Each element of $\mathbb C[\Gamma]$ has a norm defined as:
\[
	\left\|\sum_{\gamma\in\Gamma}a_{\gamma}\gamma\right\|_1=\sum_{\gamma\in\Gamma}|a_{\gamma}|.
\]
The space 
\[
l^1(\Gamma)=\{f:\Gamma\to\mathbb C: \|f\|_1=\sum_{\gamma\in\Gamma}|f(\gamma)|<\infty\}	
\]
of summable functions on $\Gamma$,
is a convolution algebra for the law
\[
	f*g(x)=\sum_{\gamma\in\Gamma}g(\gamma^{-1}x)f(\gamma),\, \forall x\in\Gamma.
\]
There is a unique embedding of involutive unital algebras of $\mathbb C[\Gamma]$ into $l^1(\Gamma)$, which sends each $\gamma\in\mathbb C[\Gamma]$ to the characteristic function $\delta_{\gamma}\in l^1(\Gamma)$ of $\gamma$.
Let $\pi:\Gamma\to U({\mathcal H})$ be a unitary representation of $\Gamma$ on a Hilbert space ${\mathcal H}$. Let $B({\mathcal H})$ be the involutive algebra of bounded operators on ${\mathcal H}$. If $T\in B({\mathcal H})$ we denote by $\|T\|$ its operator norm. There is a  unique  morphism of involutive unital algebras from $l^1(\Gamma)$ to $B({\mathcal H})$ whose restriction to $\Gamma$ equals $\pi$. We also denote this morphism by $\pi$.

Let $E$ be a finite subset of $\Gamma$ and let $|E|$ be its cardinality. We denote the element of $\mathbb C[\Gamma]$ defined as the sum over elements of $E$ by
\[
	\ungra_E=\sum_{\gamma\in E}\gamma
\]and we define $\mu_E$ to be $\frac{1}{\vert E \vert}\ungra_E$.

Let $l^2(\Gamma)$ be the Hilbert space of square integrable functions on $\Gamma$. Let $\rho_{\Gamma}:\Gamma\to U(l^2(\Gamma))$ be the right regular representation: 
\[
	 \rho_{\Gamma}(\gamma)f(x)=f(x\gamma),\,\forall x,\gamma\in\Gamma.
\]

\subsection{Koopman representations}
Let $X$ be a probability space and let $\nu$ be a probability measure on $X$ without atom. Let  $G=\mbox{Aut}(X,\nu)$ be the group of all measure-preserving transformations of $X$. (Any element in $G$ is represented by a map $g:X\to X$ which is one-to-one and onto, such that $\nu(g^{-1}B)=\nu(B)$ for any $B$ in the $\sigma$-algebra on which $\nu$ is defined. Two such transformations are identified if and only if they coincide on a set of full measure.) Let ${\mathcal H}=L^2(X,\nu)$ be the Hilbert space of complex square integrable functions on $X$. If $f\in L^2(X,\nu)$ we denote its norm by $\|f\|_2$.
Let 
$\pi:G\to U(L^2(X,\nu))$ be the Koopman representation:
\[
	\pi(g)f(x)=f(g^{-1}x).
\]

\subsection{Discrepancy}
Let $1_X\in {\mathcal H}$ be the characteristic function of $X$. Let $P\in B({\mathcal H})$ be the orthogonal projection onto the complex line generated by $1_X$. We have $P^2=P$, $P=P^*$, $P\pi(g)=\pi(g)P=P$ for all $g\in G$. Let $\pi_0$ denote the restriction of $\pi$ to the kernel 
\[
	\mbox {Ker} P=L_0^2(X,\nu)=\{f\in L^2(X,\nu): \int_Xf(x)d\nu(x)=0\}.
\]
\begin{defi} Let $E=E^{-1}$ be a finite subset of $G$. The discrepancy of $E$ is defined as the norm $\|\pi_0(\mu_E)\|$ of the operator
$\pi_0(\mu_E):L_0^2(X,\nu)\to L_0^2(X,\nu)$.
\end{defi}

It is easy to check that 
\begin{align}\label{formula: discrepancy}
\|\pi_0(\mu_E)\|=\sup_{\|f\|_2=1}\left\|\frac{1}{|E|}\sum_{\gamma\in E}f(\gamma x)-\int_Xf(y)d\nu(y)\right\|_2,
\end{align}
where the supremum is taken  over $f$ running in the unit sphere of the whole Hilbert space $L^2(X,\nu)$.


\section{Proofs}

To prove the results stated in the introduction, we start with a general lower bound  on the discrepancy which is sharp in the case of the isometries of the sphere, as well as in the case of the automorphisms of the torus. We then prove the upper bounds in the case of the sphere. Finally we prove the upper bounds in the case of the torus. 
\subsection{A lower bound for the discrepancy}

We present two  proofs of the following lower bound which generalizes and strengthens  \cite[Th\'eor\`eme 3.3]{Sev}, \cite[Theorem 2]{Clo}, \cite[Theorem 4]{Pis}, \cite[Theorem 1.3]{LPS-1}. The first proof is based on a stronger result due to Shalom \cite[Theorem 4.14]{ShaTata} which also appears in \cite[Proposition 14]{DudGri}. The second proof we give is short and self-contained. 

\begin{theo}\label{theo: lower bound on discrepancy} Let $(X,\nu)$ be an atomless probability space. Let $\Gamma$ be a finitely generated group of measure-preserving transformations of $X$. Let $m\in\mathbb C[\Gamma]$ be a positive element (i.e. $m=\sum_{\gamma\in\Gamma} a_{\gamma}\gamma$ with $a_{\gamma}\geq 0$). Then:
	\[
		\|\pi_0(m)\|\geq\|\rho_{\Gamma}(m)\|.
	\]
\end{theo}

\begin{rema} Let $\mathcal{H}$ be a Hilbert space and let $\pi:\Gamma\to U(\mathcal{H})$ be a unitary representation. The unique morphism $\pi:l^1(\Gamma)\to B(\mathcal{H})$ of unital involutive algebras extending $\pi$ is continuous.  (More precisely it satisfies $\|\pi(f)\|\leq\|f\|_1$ for all $f\in l^1(\Gamma)$.) As a consequence, Theorem \ref{theo: lower bound on discrepancy} extends to positive elements in $l^1(\Gamma)$.
\end{rema}

\begin{rema}\label{rem: Pierre-Emmanuel} As mentioned to us by P.-E. Caprace, the existence of measure preserving actions, of locally compact \emph{amenable} groups, with a spectral gap (see for example \cite[Corollary 1.10, 1, Chap. III]{Mar}), prevents an obvious generalization of the above result to locally compact groups. But, we believe the following statement is true. Let $\mu$ be a Haar measure on a locally compact group $G$. Let $\rho_G$ be the right-regular representation of $G$. Let $X$ be a Hausdorff topological space. Let $\nu$ be an atomless Borel regular probability measure on $X$. Suppose that $G$ acts continuously on $X$, and preserves $\nu$. Assume there exists a point $x_0$ in the support of $\nu$, with compact stabilizer, such that the orbit of $x_0$, under  any compact subset $K$ of $G$, has zero measure: $\nu(Kx_0)=0$.  Let $\pi_0:L^1(G,\mu)\to B(L^2_0(X,\nu))$ be the Koopman representation, restricted to the subspace of functions with zero integral. Then, for any continuous function $f:G\to \mathbb R$, with compact support, and taking only nonnegative values, we have:
		\[
			\|\pi_0(f)\|\geq\|\rho_G(f)\|.
		\] 
And if $G$ is is second countable, the inequality extends to nonnegative functions belonging to $L^1(G,\mu)$. (We hope to say more on these questions in a forthcoming paper.) 
\end{rema}	

First proof of Theorem \ref{theo: lower bound on discrepancy}. 
\begin{proof} According to \cite[Theorem 4.14]{ShaTata} or \cite[Proposition 7]{DudGri}, there exists a subgroup $H$ of $\Gamma$ such that the quasi-regular representation of $\Gamma$ on $l^2(\Gamma/H)$ is weakly contained in the restriction of the  Koopman representation of $\Gamma $  to the orthogonal complement $L^2_0(X,\nu)$ of the constant functions. As the quasi-regular representation of $\Gamma$ on $l^2(\Gamma/H)$ contains positive vectors,
the theorem follows from \cite[Lemma 2.3]{ShaAnnals}, and the definition of weak containment \cite[Definition F.1.1]{BdlHV}.
\end{proof}

Second proof of Theorem \ref{theo: lower bound on discrepancy}.

\begin{proof} If $m=0$ the inequality is trivial. If $m\neq 0$, multiplying $m$ by $\|m\|_1^{-1}$, we may assume that $\|m\|_1=1$. 
As $\|T\|^2=\|TT^*\|$ for any bounded operator $T$ on a Hilbert space, we may moreover assume $m=m^*$.

Let $e\in\Gamma$ be the identity element. We claim that for any $n\in\mathbb N$, $\|\pi_0(m^n)\|\geq m^{(n)}(e)$, where $m^{(n)}\in l^1(\Gamma)$ is the $n$-th convolution power of $m$ (although $m$ has finite support it is convenient here and in what follows to view $m$ in the convolution algebra $l^1(\Gamma)$ of summable functions). 

To prove this claim, let $F$ be the support of $m^{(n)}$. We choose  a measurable subset $B_+$ of $X$ such that 
\[
0<\nu(B_+)<\frac{1}{2|F|}.	
\]
Such a set obviously exists because $\nu$ is finite without atom. As the action preserves the measure,
\[
	\nu(FB_+)\leq|F|\nu(B_+)<1/2.
\]
A finite atomless measure has the intermediate value property (see \cite[1.12.10 Corollary]{Bog}) hence there exists a measurable subset $B_-$ of $X\setminus FB_+$
satisfying  $\nu(B_+)=\nu(B_-)$. Let
\[
	\varphi=\frac{1_{B_+}-1_{B_-}}{\|1_{B_+}-1_{B_-}\|_2}\in L_0^2(X,\nu).
\]
The proof of the claim follows then from the Cauchy-Schwarz inequality, the symmetry of $F$, and the positivity of $m$:
\begin{align*}
	\|\pi_0(m^n)\|&\geq\langle\pi_0(m^n)\varphi,\varphi\rangle\\
	              &=\frac{1}{\nu(B_+)+\nu(B_-)}\sum_{\gamma\in \Gamma}[\nu\left(\gamma B_+\cap B_+\right)+\nu(\gamma
				  B_-\cap B_-)]m^{(n)}(\gamma)\\
                  &\geq \frac{1}{\nu(B_+)+\nu(B_-)}\sum_{\gamma=e}[\nu\left(\gamma B_+\cap B_+\right)+\nu(\gamma B_-\cap B_-)]m^{(n)}(\gamma)\\
	              &=m^{(n)}(e).
\end{align*} 
This finishes the proof of the claim.

Applying the claim we obtain:
\begin{align*}
	\|\pi_0(m)\|&=\lim\limits_{n\to\infty}\, \|\pi_0(m)^n\|^{1/n}\\
	&\geq\limsup_{n\to\infty}m^{(n)}(e)^{1/n}\\
	&=\|\rho_{\Gamma}(m)\|.
	\end{align*}
The last equality above goes back to \cite[Lemma 2.2]{Kes}.	 	
\end{proof}

The following  corollary together with Proposition \ref{prop: H-Ch lower bound on the discrepancy for free groups}  illustrate Theorem \ref{theo: lower bound on discrepancy} with two opposite situations: the discrepancy is maximal in the case the transformations generate an amenable group whereas the discrepancy may be small in the case the transformations generate a free group.

\begin{coro}\label{cor: amenable} Let $E\subset G$ be finite and symmetric. Let $\Gamma$ be the group generated by $E$. If $\Gamma$ is amenable then
	\[
		\sup_{\|f\|_2=1}\left\|\frac{1}{|E|}\sum_{\gamma\in E}f(\gamma x)-\int_Xf(y)d\nu(y)\right\|_2=1.
	\]
\end{coro}

\begin{proof} Without any hypothesis on $\Gamma$ we always have $\|\pi_0(\mu_E)\|\leq \|\mu_E\|_1=1$. According to Theorem \ref{theo: lower bound on discrepancy},
$\|\pi_0(\mu_E)\|\geq\|\rho_{\Gamma}(\mu_E)\|$. According to \cite{KesFull}, the group $\Gamma$ is amenable if and only if $\|\rho_{\Gamma}(\mu_E)\|=1$.	 
\end{proof}

We prove Proposition \ref{prop: H-Ch lower bound on the discrepancy for free groups}.
\begin{proof} Let $n\in\mathbb N$ be given. Let $S_n\subset\Gamma$ be the sphere around $e$ of radius $n$.
As explained in Proposition \ref{prop: computing the norm of the regular representation} from the Appendix, or according to \cite{Coh} or \cite[12.17]{Woe}, 
\[
\|\rho_{\Gamma}(\mu_{S_n})\|=\left(1+\frac{q-1}{q+1}n\right)q^{-n/2}.	
\]
Let $E=S_n$ and let $H<\Gamma$ be the subgroup generated by $E$. Let $\mu_{E}\in\mathbb C[H]\subset\mathbb C[\Gamma]$. Decomposing $\Gamma$ into its right $H$-cosets, it is easy to check that $\rho_{\Gamma}$ restricted to $H$ decomposes as a direct sum of unitary representations, all unitary equivalent to $\rho_{H}$, and that consequently 
	\[
\|\rho_{\Gamma}(\mu_E)\|=\|\rho_{H}(\mu_E)\|.
	\] 
Applying Theorem \ref{theo: lower bound on discrepancy} we obtain
\[
	\|\pi_0(\mu_E)\|\geq\|\rho_{H}(\mu_E)\|.
\]
Applying Formula \ref{formula: discrepancy}	 finishes the proof of the corollary in the case $E=S_n$.
The case of the ball of radius $n$ is similar.
\end{proof}

\subsection{Exact convergence rate for some isometries of the sphere}
We recall the construction from \cite{LPS-2} of free subgroups of isometries of the round sphere.
Let $\mathbb H=\{q=x_0+x_1i+x_2j+x_3k: x_0,x_1,x_2,x_3\in\mathbb R\}$ be the field of quaternions. Let 
$$\tau(q)=\overline{x_0+x_1i+x_2j+x_3k}=x_0-x_1i-x_2j-x_3k$$
denote the conjugate of $q$. Let
$N(q)=q\overline{q}$ be the norm of $q$ and let $|q|=\sqrt{N(q)}$ be its module. The multiplicative group $\mathbb H^*$ acts on $\mathbb H$ by conjugation and  if $q\in\mathbb H^*$ and $v\in\mathbb H$, then $|qvq^{-1}|=|v|$. As this action preserves the subspace 
$\mbox{Im}\,\mathbb H=\{x_1i+x_2j+x_3k: x_1,x_2,x_3\in\mathbb R\}$, it defines a homomorphism
\[
	\mbox{Ad}:\mathbb H^*\to \SO(3,\mathbb R),\, q\mapsto(v\mapsto qvq^{-1})
\]
with values in the orientation preserving isometry group of the round sphere $\sss^2$.
The ring
\[
	\mathbb H(\mathbb Z)=\{q=x_0+x_1i+x_2j+x_3k: x_0,x_1,x_2,x_3\in\mathbb Z\} 
\]
of Lipschitz quaternions has $8$ units:
\[
\mathbb H(\mathbb Z)^{\times}=\{\pm 1,\pm i,\pm j\pm k\}. 	
\]
Let $n\in\mathbb N$. According to Jacobi (see for example \cite[p. 27]{Zag} or \cite[Theorem 2.4.1]{Val} for odd integers), the cardinality of the set of Lipschitz quaternion of norm $n$ is
\[
	|N^{-1}(n)\cap\mathbb H(\mathbb Z)|=8\sum_{4\nmid d\mid n}d.
\]
Hence, if $n=p$ is prime, the set  $N^{-1}(p)\cap\mathbb H(\mathbb Z)$ splits as the disjoint union of the $p+1$ orbits of the action of $\mathbb H(\mathbb Z)^{\times}$.
In the case $p\equiv1\mod 4$, it is easy to check that each orbit contains a unique quaternion $q=x_0+x_1i+x_2j+x_3k$ with $x_0>0$ and $x_0\equiv1\mod 2$ and that
the set 
\[
\{q=x_0+x_1i+x_2j+x_3k: x_0,x_1,x_2,x_3\in\mathbb Z,\, N(q)=p,\, x_0>0,\, x_0\equiv1\mod 2\}
\]
splits into $\frac{p+1}{2}$ orbits of the involution $\tau$, each containing two elements. Let $\Sigma_{p+1}\subset \SO(3,\mathbb R)$ denote the image of this set under the homomorphism $\mbox{Ad}$.

We are ready to prove Theorem \ref{thm: main}.
\begin{proof} It follows from \cite[S153-S158]{LPS-1} and \cite[Theorem 4.1]{LPS-2}, that the spectrum of $\pi_0$
 satisfies	
\begin{align}\label{formula: inclusion between spectra}
 \sigma\left(\pi_0\left(\ungra_{\Sigma_{p+1}}\right)\right)\subset[-2 \sqrt p,2 \sqrt p].
\end{align}
(In fact it is shown in \cite[S153-S158]{LPS-1} that $\sigma\left(\pi_0\left(\ungra_{\Sigma_{p+1}}\right)\right)=[-2 \sqrt p,2 \sqrt p]$. We  will ``only'' use the inclusion $\sigma(\pi_0(\ungra_{\Sigma_{p+1}}))\subset[-2 \sqrt p,2 \sqrt p]$ but this is by far the hardest to prove; the main ingredient in its proof is the inequality \cite[Theorem 4.1]{LPS-2} which relies in particular on \cite{Del}.)
Applying Inclusion \ref{formula: inclusion between spectra} and Theorem \ref{theo: lower bound on discrepancy} we deduce the inequalities
\[2 \sqrt p\geq\left\|\pi_0\left(\ungra_{\Sigma_{p+1}}\right)\right\|\geq\left\|\rho_{\Gamma}\left(\ungra_{\Sigma_{p+1}}\right)\right\|.
\]
It then follows from Kesten's spectral characterization of free groups (see \cite[S157]{LPS-1} and \cite{Kes})  that $\Gamma$ is free of rank $r=\frac{p+1}{2}$ (and freely generated by any subset $A$ of $\Sigma_{p+1}$ containing $\frac{p+1}{2}$ elements and satisfying $A\cap A^{-1}=\emptyset$).
On $\Gamma$ we consider the word metric defined by $\Sigma_{p+1}$, and for each $n\in\mathbb N\cup\{0\}$, the Hecke element
\[
	T_n=\sum_{|\gamma|=n}\gamma\in\mathbb C[\Gamma].
\]
We have: $T_0=e$, $$T_1=\sum_{\gamma\in \Sigma_{p+1}}\gamma = \ungra_{\Sigma_{p+1}},$$ $$T_1T_1=T_2+2rT_0.$$ If $n\geq 2$ we have: 
\[
T_nT_1=T_{n+1}+pT_{n-1}.	
\]
There is a unique morphism of unital rings, from the ring $\mathbb Z[X]$ of polynomials in one variable  with integer coefficients, to $\mathbb C[\Gamma]$,
sending $X$ to $T_1$. The above recursion relations show that $T_n$ is in the image of this morphism for any $n\geq 0$. In other words for each $n\geq 0$, there exists $P_n\in \mathbb Z[X]$ such that $T_n=P_n(T_1)$. 

We first prove the theorem in the case of a sphere $S_n\subset\Gamma$. The lower bound on the discrepancy follows from Proposition \ref{prop: H-Ch lower bound on the discrepancy for free groups}. For the upper bound, applying the spectral theorem for bounded self-adjoint operators, Inclusion \ref{formula: inclusion between spectra}, and Kesten's computation
\cite{Kes} of the spectrum of the regular representation
\[
	\sigma\left(\rho_{\Gamma}(T_1)\right)=[-2 \sqrt p,2 \sqrt p],
\]
we deduce that
\begin{align*}\label{formula: spectral theorem}
	\|\pi_0(T_n)\|&=\|\pi_0(P_n(T_1))\|\\
	              &=\|P_n(\pi_0(T_1))\|\\
                  &=\sup_{\lambda\in\sigma(\pi_0(T_1))}|P_n(\lambda)|\\
				  &\leq\sup_{\lambda\in[-2 \sqrt p,2 \sqrt p]}|P_n(\lambda)|\\
	              &=\sup_{\lambda\in\sigma(\rho_{\Gamma}(T_1))}|P_n(\lambda)|\\
				  &=\|\rho_{\Gamma}(T_n)\|.
\end{align*}
Applying Proposition \ref{prop: computing the norm of the regular representation} or \cite[12.17]{Woe}, we conclude that
$$\sup_{\|f\|_2=1}\left\|\frac{1}{|S_n|}\sum_{\gamma\in S_n}f(\gamma x)-\int_{\sss^2}f(y)d\nu(y)\right\|_2\leq \left(1+\frac{p-1}{p+1}n\right)p^{-n/2}.$$
The proof, in the case of a ball $B_n$, follows exactly the same lines.
\end{proof}

\subsection{Exact convergence rate for automorphisms of the torus}
We prove Theorem \ref{thm: exact convegence rate on the torus}.
\begin{proof} According to Proposition \ref{prop: H-Ch lower bound on the discrepancy for free groups}, the lower bounds on the discrepancies are true and the cases with $q=1$ have already been discussed in Corollary \ref{cor: amenable} and Remark \ref{rema: amenable}. As explained in \cite[20]{dlH} or in \cite[Theorem 4.17]{ShaTata}, the restriction of the Koopman representation defined by the action of $G$ on $\ttt^2$
	\[
		\pi_0:G\to U(L_0^2(\ttt^2,\nu))
	\]
is weakly contained in the regular representation $\rho_G$. Hence, according to Theorem \cite[Theorem 7]{dlH}, for any $m\in\mathbb C[G]$, 
\[
	\|\pi_0(m)\|\leq \|\rho_G(m)\|.
\]	
Choosing $m=\mu_{E_n}\in\mathbb C[\Gamma]\subset\mathbb C[G]$, where $E_n$ is either a sphere or a ball, we get (as explained in the proof of Proposition \ref{prop: H-Ch lower bound on the discrepancy for free groups}):
\[
\|\rho_G(\mu_{E_n})\|=\|\rho_{\Gamma}(\mu_{E_n})\|.	
\]
Applying Proposition \ref{prop: computing the norm of the regular representation} or \cite[12.17]{Woe} finishes the proof of the theorem.
\end{proof}

\section{Appendix}
The aim of this appendix is to recall  well-known facts about the regular and quasi-regular representations of the automorphism group of a regular tree.
 Most relevant for this paper are explicit formulae for the norms of Markov operators, defined by the regular representation of the automorphism group of the tree.
Although the formulae from Proposition \ref{prop: computing the norm of the regular representation} below follow from \cite{Coh}, or \cite[12.17]{Woe}  which is based on \cite[12.10]{Woe} (see also \cite{SalWoe} for a more general setting), it seems  worthwhile to emphasize that these formulae can also be deduced and expressed with the help of the Harish-Chandra function of the quasi-regular representation of the automorphism group of an homogeneous tree. 
(Both approaches are equivalent; the modular functions of cocompact amenable subgroups in \cite{SalWoe} correspond to Radon-Nikodym cocycles of quasi-regular representations.)

\subsection{The boundary of a tree and Busemann cocycles}
Let $(X,d)$ be the regular tree of degree $q+1$ equipped with its geodesic path metric $d$ for which each edge is isometric to the unit interval $[0,1]\subset\mathbb R$. Let $x_0$ be a vertex of $X$. Let $\partial X$ be its boundary at infinity (we refer the reader to \cite{Bou} for more details).
 Let $b\in \partial X$ and let $\beta:[0,\infty)\to X$ be a geodesic ray representing $b$. Let $x, y\in X$. Let
\[
	B_b(x,y)=\lim_{t\to\infty}[d(x,\beta(t))-d(y,\beta(t))],
\]
be the Busemann cocycle defined by $b\in\partial X$.
Let $a, b\in\partial X$ and let $\alpha,\beta:[0,\infty)\to X$ be geodesic rays representing $a$ and $b$. Their Gromov product relative to the base point $x_0$ is defined as
\[
	(a|b)_{x_0}=\frac{1}{2}\lim_{t\to\infty}[d(x_0,\alpha(t))+d(x_0,\beta(t))-d(\alpha(t),\beta(t))].
\]
The formula
\[
	d_{x_0}(a,b)=e^{-(a|b)_{x_0}}
\]
defines an ultra-metric on $\partial X$. 
\subsection{Conformal transformations and Radon-Nikodym derivatives}
The group $\mbox{Aut}(X)$ of isometries of $X$ acts on $\partial X$ by conformal transformations.
The Hausdorff dimension of $(\partial X,d_{x_0})$ equals $\log q$ and the normalized Hausdorff measure $\nu$ on $(\partial X,d_{x_0})$ is the unique Borel
probability measure on $\partial X$ invariant under the action of the stabilizer $K=\mbox{Aut}(X)_{x_0}$ of $x_0$. The Radon-Nikodym derivative of $g\in \mbox{Aut}(X)$
at $b\in\partial X$ is
\[
	\frac{dg_*\nu}{d\nu}(b)=q^{B_b(x_0,gx_0)}.
\]
\subsection{The Koopman representation and the Harish-Chandra function}
The Koopman representation 
\[
	\lambda_{\nu}:\mbox{Aut}(X)\to U(L^2(\partial X,\nu))
\]
is defined as
\[
	(\lambda_{\nu}(g)f)(b)=f(g^{-1}b)\sqrt{\frac{dg_*\nu}{d\nu}(b)}.
\]
(The representation $\lambda_{\nu}$ is unitary equivalent to the quasi-regular representation $\lambda_{G/P}$, where $G=\mbox{Aut}(X)$ and $P=G_b$, with $b\in\partial X$ any base point at infinity.)

Let $1_{\partial X}\in L^2(\partial X,\nu)$ be the constant function equal to $1$. 
The Harish-Chandra function 
\[
	\Xi:\mbox{Aut}(X)\to(0,\infty)
\]
is the coefficient of $\lambda_{\nu}$ defined by $1_{\partial X}$ that is:
\[
\Xi(g)=\langle \lambda_{\nu}(g)1_{\partial X},1_{\partial X}\rangle.
\]
As the action of $K$ preserves the measure and as $\lambda_{\nu}$ is unitary, the Harish-Chandra function is $K$-bi-invariant and symmetric. (In \cite[Part II, 16]{HarCha} Harish-Chandra introduces the function $\Xi$ on a connected reductive Lie group. The function $\Xi$ can be viewed as the spherical function  associated to a quasi-regular representation. The definitions make sense for $G$ a locally compact (separable) unimodular group with a compact subgroup $K$ such that $(G,K)$ is a Gelfand pair, and an irreducible unitary representation $\pi$ with one-dimensional $K$-invariant subspace \cite[1.5]{GanVar}. In our case $G=\mbox{Aut}(X)$, $K=\mbox{Aut}(X)_{x_0}$ and $\pi=\lambda_{\nu}$.)
   
\subsection{A formula for the Harish-Chandra function}
The length function
\[
	L:\mbox{Aut}(X)\to\mathbb N\cup\{0\}
\]
\[
	L(g)=d(x_0,gx_0)
\]
is also $K$-bi-invariant and symmetric (notice that the elements of length $0$ are the elements of $K$). As $K$ acts transitively on each sphere of $X$ with center $x_0$, if $g,g'\in \mbox{Aut}(X)$ satisfy $L(g)=L(g')$ then there exist $k,k'\in K$ such that $kg=g'k'$. This implies that $\Xi$ is constant on the level sets of $L$. For each $n\in\mathbb N\cup\{0\}$,
we will write $\Xi(n)$ for the common value of the Harish-Chandra function  on all $g\in \mbox{Aut}(X)$ such that $L(g)=n$. 
We claim that for any $n\in\mathbb N\cup\{0\}$,
\begin{align}\label{formula: Harish-Chandra}
\Xi(n)=\left(1+\frac{q-1}{q+1}n\right)q^{-n/2}.
\end{align}
If $L(g)=0$, that is if $g\in K$, then $\Xi(g)=1$. If $L(g)>0$, let
$[x_0,gx_0]$ be the (image of the) unique geodesic segment of $X$ between $x_0$ and $gx_0$. For each vertex $x$ of $X$ at distance exactly $1$ from $[x_0,gx_0]$, consider the ball
$U_x$ of $\partial X$ of radius $e^{-d(x_0,x)}$ consisting of the points at infinity of the geodesic rays of $X$ starting from $x_0$ and passing through $x$. We obtain the partition
\[
	\partial X=\bigcup_{d(x,[x_0,gx_0])=1}U_x.
\]
As for each $r\in\mathbb N$ the measure of a ball of radius $e^{-r}$ equals $[(q+1)q^{r-1}]^{-1}$, it is easy to check that
\begin{align*}
\Xi(g)&=\int_{\partial X}q^{\frac{1}{2}B_b(x_0,gx_0)}d\nu(b)\\
      &=\sum_{d(x,[x_0,gx_0])=1}\int_{U_x}q^{\frac{1}{2}B_b(x_0,gx_0)}d\nu(b)\\
	  &=\sum_{d(x,[x_0,gx_0])=1}\int_{U_x}q^{\frac{1}{2}(d(x_0,x)-d(gx_0,x))}d\nu(b)\\
	  &=\sum_{d(x,[x_0,gx_0])=1}q^{\frac{1}{2}(d(x_0,x)-d(gx_0,x))}\nu(U_x)\\
      &=\left(1+\frac{q-1}{q+1}n\right)q^{-n/2}.
\end{align*}
\subsection{Computing operator norms with the Harish-Chandra function}
We first compute the norms of some operators defined by the Koopman representation. We then explain how  spectral transfer applies to compute the norms of the corresponding operators defined by the regular representation.
\begin{prop}\label{prop: computing the norm with the Harish-Chandra function} Let $r\in\mathbb N$. Let $\Gamma$ be the free group of rank $r$. Let $a_1,\cdots, a_r$ be a free generating set of $\Gamma$. Let $X$ be the Cayley graph of $\Gamma$ with respect to $\{a_1^{\pm 1},\cdots, a_r^{\pm 1}\}$. Let $e=x_0\in X$ be a base point. Let $G=\mbox{Aut}(X)$ and let $L(g)=d(x_0,gx_0)$ be the length function on $G$ defined by $x_0$. For each integer $n\geq 0$, let $S_n=L^{-1}(n)\cap\Gamma$ and let $B_n=L^{-1}([0,n])\cap\Gamma$ (where $\Gamma\subset G$ is the natural embedding). Then
	\[
		\|\lambda_{\nu}(\mu_{S_n})\|=\Xi(n), 
	\]
	\[
		\|\lambda_{\nu}(\mu_{B_n})\|=\frac{1}{|B_n|}\sum_{k=0}^n\Xi(k)|S_k|.
	\]
\end{prop}
\begin{proof} We first consider the case of the sphere $S_n$. Let $1_{\partial X}$ be the constant function equal to $1$ on $\partial X$. Applying the Cauchy-Schwarz inequality and the fact that the function $\Xi$ is constant on $S_n$, we obtain:
\begin{align*}
\|\lambda_{\nu}(\mu_{S_n})\|&\geq \|\lambda_{\nu}(\mu_{S_n})1_{\partial X}\|_2=\|\lambda_{\nu}(\mu_{S_n})1_{\partial X}\|_2\|1_{\partial X}\|_2\\
                            &\geq\langle\lambda_{\nu}(\mu_{S_n})1_{\partial X},1_{\partial X}\rangle=
							\frac{1}{|S_n|}\sum_{\gamma\in S_n}\langle\lambda_{\nu}(\gamma)1_{\partial X},1_{\partial X}\rangle\\
							&=\Xi(n).
\end{align*}

To prove the other inequality, we first notice that for $p\in\{1,2,\infty\}$, the operator $\lambda_{\nu}(\mu_{S_n}):L^p(X,\nu)\to L^p(X,\nu)$ is bounded. In the case $p=2$ it is self-adjoint.  Thanks to  Riesz-Thorin's theorem,
\[
\|\lambda_{\nu}(\mu_{S_n})\|_{2\to 2}\leq \|\lambda_{\nu}(\mu_{S_n})\|_{\infty\to \infty}.
\]
As $\lambda_{\nu}(\mu_{S_n})$ preserves positive functions, it is obvious that $$\|\lambda_{\nu}(\mu_{S_n})\|_{\infty\to \infty}=\|\lambda_{\nu}(\mu_{S_n})1_{\partial X}\|_{\infty}.$$
We claim that the function $\lambda_{\nu}(\mu_{S_n})1_{\partial X}$ is constant equal to $\Xi(n)$. To prove the claim we first show that the function is $K$-invariant (hence constant as $K$ acts transitively on $\partial X$). Let $b\in\partial X$ and $k\in K$. We have:
\begin{align*}
\lambda_{\nu}(\mu_{S_n})1_{\partial X}(kb)&=\frac{1}{|S_n|}\sum_{\gamma\in S_n}q^{\frac{1}{2}B_{kb}(x_0,\gamma x_0)}\\
                                          &=\frac{1}{|S_n|}\sum_{\gamma\in S_n}q^{\frac{1}{2}B_{b}(k^{-1}x_0,k^{-1}\gamma x_0)}\\
                                          &=\frac{1}{|S_n|}\sum_{\gamma\in S_n}q^{\frac{1}{2}B_{b}(x_0,k^{-1}\gamma x_0)}\\
                                          &=\frac{1}{|S_n|}\sum_{\gamma\in S_n}q^{\frac{1}{2}B_{b}(x_0,\gamma x_0)}\\
										  &=\lambda_{\nu}(\mu_{S_n})1_{\partial X}(b).
\end{align*}									  
The claim is proved because $\nu$ is a probability measure and by definition of $\Xi$ we have:
\[
	\int_{\partial X}\lambda_{\nu}(\mu_{S_n})1_{\partial X}(b)d\nu(b)=\Xi(n).
\]
									  
In the he case of the ball $B_n$, applying Cauchy-Schwarz's inequality and Riesz-Thorin's theorem in similar ways proves  that
	\[
		\left\|\lambda_{\nu}\left(\ungra_{B_n}\right)\right\|=\sum_{\gamma\in B_n}\Xi(\gamma).
	\]
The result follows because $\Xi$ is constant on spheres.									  
\end{proof}

\begin{prop}\label{prop: computing the norm of the regular representation} Let $r\in\mathbb N$. Let $\Gamma$ be the free group of rank $r$. Let $a_1,\cdots, a_r$ be a free generating set of $\Gamma$. Let $S=\{a_1^{\pm 1},\cdots, a_r^{\pm 1}\}$. For each integer $n\geq 0$, let $S_n$, respectively $B_n$, be the sphere, respectively the ball, around $e\in\Gamma$ of radius $n$ with respect to the word metric on $\Gamma$ defined by $S$. Let $q=2r-1$. Then
	\[
		\|\rho_{\Gamma}(\mu_{S_n})\|=\left(1+\frac{q-1}{q+1}n\right)q^{-n/2}, 
	\]
\[
		\|\rho_{\Gamma}(\mu_{B_n})\|=c(q,n)\left(1+\left(1+\frac{1}{\sqrt q}\right)n\right)q^{-n/2},
\]
where $c(q,n)=\left(1+2q^{-n}\sum_{k=0}^{n-1}q^k\right)^{-1}$.
\end{prop}
\begin{proof} We claim that for any positive element $m\in\mathbb C[\Gamma]$, we have 
\[
	\|\rho_{\Gamma}(m)\|=\|\lambda_{\nu}(m)\|.
\]
The inequality $\|\rho_{\Gamma}(m)\|\leq \|\lambda_{\nu}(m)\|$ follows from \cite[Lemma 2.3]{ShaAnnals} because $1_{\partial X}$ is a positive vector for $\lambda_{\nu}$. The inequality
$\|\rho_{\Gamma}(m)\|\geq \|\lambda_{\nu}(m)\|$ is true for any element  $m\in\mathbb C[\Gamma]$ (not only positive ones), because the action of $\Gamma$ on $\partial X$ is amenable, see \cite{Kuh} and \cite[Theorem 7]{dlH}. To prove the proposition, we apply Proposition \ref{prop: computing the norm with the Harish-Chandra function} and Formula \ref{formula: Harish-Chandra}. In the case of the sphere $S_n$, this immediately proves the statement. The case of the ball $B_n$ requires some computation. If $q=1$, the formula is obvious.  If $q>1$, we have
\begin{align*}
\|\rho_{\Gamma}(\mu_{B_n})\|&=\frac{1}{|B_n|}\sum_{k=0}^n\Xi(k)|S_k|\\
                            &=\frac{1}{|B_n|}\left(1+\frac{q+q^{1/2}}{q}n\right)q^{n/2}\\
							&=\left(\frac{q+1}{q-1}q^n-\frac{2}{q-1}\right)^{-1}\left(1+\frac{q+q^{1/2}}{q}n\right)q^{n/2}\\
							&=c(q,n)\left(1+\left(1+\frac{1}{\sqrt q}\right)n\right)q^{-n/2},
\end{align*}	
where $c(q,n)=\left(1+2q^{-n}\sum_{k=0}^{n-1}q^k\right)^{-1}$.	
\end{proof}

\end{document}